\documentclass[a4paper,12pt]{amsart}
\pdfoutput=1
\usepackage{setspace}
\usepackage[left=35mm,head=30mm,bottom=30mm,foot=20mm,right=35mm]{geometry}
\usepackage{amsmath,amsthm,amsfonts,amssymb,mathtools,mathrsfs,url,bm,enumitem,stackengine}
\usepackage{newtxtext}
\usepackage[varvw]{newtxmath}
\allowdisplaybreaks
\usepackage{hyperref}
\hypersetup{
  colorlinks=true,
  linkcolor=blue,
  citecolor=red,
}
\mathtoolsset{showonlyrefs=true}
\theoremstyle{definition}
  
  \newtheorem{rem}{Remark}[section]

\theoremstyle{plain}
  \newtheorem{thm}{Theorem}[section]
  
  \newtheorem{prop}{Proposition}[section]
  \newtheorem{lem}{Lemma}[section]
\renewcommand{\Re}{\operatorname{Re}}

\renewcommand{\phi}{\varphi}
\newcommand\restr[2]{{
  \left.\kern-\nulldelimiterspace 
  #1 
  \vphantom{\big|} 
  \right|_{#2} 
}}

\let\phi\varphi
\let\epsilon\varepsilon

\DeclareMathOperator{\dist}{dist}

\newcommand{\relmiddle}[1]{\mathrel{}\middle#1\mathrel{}}
\DeclareMathOperator{\spn}{span}

\usepackage{tikz-cd}
\tikzset{
  symbol/.style={
    draw=none,
    every to/.append style={
      edge node={node [sloped, allow upside down, auto=false]{$#1$}}}
  }
}

\begin{document}

\begin{abstract}
    We consider a viscous compressible barotropic flow in the interval $[0,\pi]$ with homogeneous Dirichlet boundary conditions for the flow velocity and a constant rest state as initial data. Given two sufficiently close subintervals $I=[\alpha_1,\alpha_2]$ and $J=[\beta_1,\beta_2]$ of $(0,1)$, a nonempty open set $\omega \subset (1,\pi)$, and $T>0$, we construct an external force $f$ supported in $\omega$ acting on the momentum equation such that the corresponding flow map moves the fluid particles initially occupying $I$ exactly onto $J$ in time $T$.
\end{abstract}

\title[Local exact Lagrangian controllability]{Local exact Lagrangian controllability for 1D barotropic compressible Navier--Stokes equations}

\author{Kai Koike}
\address{Department of Mathematics, Institute of Science Tokyo, 2-12-1, Ookayama, Meguro-ku, Tokyo 152-8551, Japan}
\email{koike.k@math.titech.ac.jp}

\author{Franck Sueur}
\address{Department of Mathematics, 
Maison du nombre, 6 avenue de la Fonte, 
University of Luxembourg,  
L-4364 Esch-sur-Alzette, Luxembourg}
\email{franck.sueur@uni.lu}

\author{Gast\'{o}n Vergara-Hermosilla}
\address{Institute for Theoretical Sciences, Westlake University, 600 Dunyu Road, 310030, Hangzhou, Zhejiang, People’s Republic of China
}
\email{gvh@westlake.edu.cn}


\maketitle

\section{Introduction}
Let a viscous compressible fluid occupy the interval $[0,\pi]$. Denote by $\rho=\rho(t,x)$ and $u=u(t,x)$ the density and the velocity of the fluid, respectively. We assume that the flow is barotropic, that is, the pressure $p=p(\rho)$ is a function only of $\rho$. Moreover, we assume a physically natural condition that $p$ is smooth around $\rho=1$ and that $p'(1)>0$. Assuming also that the fluid is initially at rest and cannot escape from both ends of $[0,\pi]$, the time evolution of $(\rho,u)$ can be described by the following barotropic compressible Navier--Stokes equations:
\begin{equation}
  \label{eq:model}
  \begin{dcases}
    \rho_t+(\rho u)_x=0 & \text{in $(0,T)\times (0,\pi)$}, \\
    \rho(u_t+uu_x)+p(\rho)_x=u_{xx}+f & \text{in $(0,T)\times (0,\pi)$}, \\
    u(t,0)=0, \quad u(t,\pi)=0 & \text{for $t\in (0,T)$}, \\
    (\rho,u)(0,x)=(1,0) & \text{for $x\in (0,\pi)$}.
  \end{dcases}
\end{equation}
Here $T$ is a positive real number and the viscous coefficient in front of $u_{xx}$ is set to unity for simplicity. More importantly, $f\in L^2((0,T)\times (0,\pi))$ is an external force whose support in $x$ is included in a nonempty open set $\omega \subset (1,\pi)$.

Associated to the solution $(\rho,u)$ to~\eqref{eq:model}, which is determined by $f$, we define the flow map
\begin{equation}
    \phi=\phi[f]\colon [0,T]\times [0,\pi]\to [0,\pi]
\end{equation}
by $\phi(0,x)=x$ and $\phi_t(t,x)=u(t,\phi(t,x))$. That this ODE is uniquely solvable (in its integral form) is guaranteed by the regularity $u\in L^1(0,T;\, \mathrm{Lip}([0,\pi]))$ obtained from Proposition~\ref{prop:wellposedness} below,\footnote{Here $\mathrm{Lip}([0,\pi])$ is the Banach space of Lipschitz continuous functions on $[0,\pi]$ endowed with the norm $\| f \|_{\mathrm{Lip}([0,\pi])}=\sup_{x\in [0,\pi]}|f(x)|+\sup_{x,y\in [0,\pi],\, x\neq y}|f(x)-f(y)|/|x-y|$.} which can be proved similarly to~\cite[Theorem~7.1]{MN80}; note however that we only consider smooth $f$ in this paper, so that $u$ is actually smooth.

\begin{prop}
    \label{prop:wellposedness}
    For $f\in L^2((0,T)\times (0,\pi))$, there exists $\epsilon_0>0$ such that if $\epsilon \coloneqq \| f \|_{L^2((0,T)\times (0,\pi))}\leq \epsilon_0$, then~\eqref{eq:model} has a unique strong solution in the class
    \begin{gather}
        \rho \in C([0,T];\, H^1(0,\pi))\cap C^1([0,T];\, L^2(0,\pi)), \\
        u\in C([0,T];\, H_{0}^{1}(0,\pi))\cap L^2(0,T;\, H^2(0,\pi)), \quad u_t \in L^2(0,T;\, L^2(0,\pi)).
    \end{gather}
    Moreover, there exists a positive constant $C>0$ such that
    \begin{equation}
        \sup_{0\leq t\leq T}\left( \| \rho(t,\cdot)-1 \|_{H^1(0,\pi)}+\| u(t,\cdot) \|_{H^1(0,\pi)} \right)+\left( \int_{0}^{T}\| u_x(t,\cdot) \|_{H^1(0,\pi)}^{2}\, dt \right)^{1/2}\leq C\epsilon.
    \end{equation}
\end{prop}

\subsection{Local exact Lagrangian controllability}
The purpose of this paper is to investigate the \textit{local exact Lagrangian controllability} for the compressible Navier–Stokes equations~\eqref{eq:model}. That is, we ask whether one can choose an appropriate external force $f$ supported in $\omega \subset (1,\pi)$ so that the flow map $\phi=\phi[f]$ moves the fluid particles initially occupying an interval $I=[\alpha_1,\alpha_2]\subset (0,1)$ exactly onto another interval $J=[\beta_1,\beta_2]\subset (0,1)$ in a given time $T>0$; see Figure~\ref{fig:01}.

\begin{figure}[htbp]
    \centering
    \includegraphics[scale=0.85]{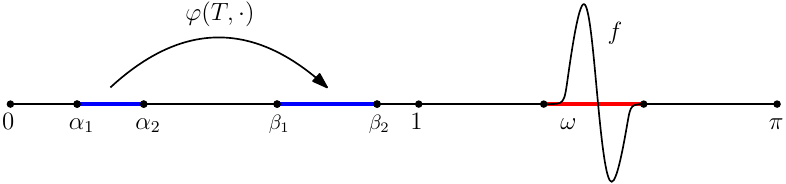}
    \caption{Can we find a remote control $f=f(t,x)$ whose support in $x$ is contained in $\omega$ such that the corresponding flow map $\phi=\phi[f]$ moves the interval $I=[\alpha_1,\alpha_2]$ exactly onto $J=[\beta_1,\beta_2]$ in a given time $T>0$?}
    \label{fig:01}
\end{figure}

Our main result, Theorem~\ref{thm:local_exact_Lagrangian_controllability} below, answers this question affirmatively when $I$ and $J$ are sufficiently close. In fact, we can achieve this using a smooth control of the form $f=\epsilon_1 f_1+\epsilon_2 f_2$ (where $\epsilon_1,\epsilon_2 \in \mathbb{R}$) with $f_1$ and $f_2$ defined by
\begin{equation}
    \label{eq:fi}
    f_i(t,x)=\chi_{\eta}(x)\xi_i(t,x) \quad (i=1,2).
\end{equation}
Here $\chi_{\eta}$ ($\eta>0$) is a non-negative smooth function supported in $\omega$ such that\footnote{Instead of simply setting $f_i(t,x)=\bm{1}_{\omega}(x)\xi_i(t,x)$ (where $\bm{1}_{\omega}$ is the indicator function of $\omega$), we multiplied the cutoff function $\chi_{\eta}$ to $\xi_i$. This makes $f_i$ smooth on $(0,T)\times (0,\pi)$ (see Lemma~\ref{lem:smoothness}), which we use to ensure that the map $\Theta$ in the proof is $C^1$; see the first paragraph of Section~\ref{sec:proof_Jac}.}
\begin{equation}
    \chi_{\eta}(x)=1 \quad (\dist(x,\omega^c)\geq \eta)
\end{equation}
and $(\zeta_i,\xi_i)$ is the solution to the linearized adjoint system 
\begin{equation}
  \label{eq:model_linearized_adjoint}
  \begin{dcases}
    -\partial_t \zeta_i-\partial_x \xi_i=0 & \text{in $(0,T)\times (0,\pi)$}, \\
    -\partial_t \xi_i-c^2\partial_x \zeta_i=\partial_{xx}\xi_i+\delta_{\alpha_i} & \text{in $(0,T)\times (0,\pi)$}, \\
    \xi_i(t,0)=0, \quad \xi_i(t,\pi)=0 & \text{for $t\in (0,T)$}, \\
    (\zeta_i,\xi_i)(T,x)=(0,0) & \text{for $x\in (0,\pi)$}
  \end{dcases}
\end{equation}
for $i=1,2$; here $c=\sqrt{p'(1)}>0$ and $\delta_{\alpha_i}$ is the Dirac delta function centered at $x=\alpha_i$. We note that this definition is inspired by the proof of~\cite[Theorem~3.1]{Horsin08}, which deals with an analogous problem for viscous Burgers' equation. The motivation behind the definition of $f_i$ becomes clear in the proof of Lemma~\ref{lem:J_is_Gram_det}.

The following properties of $\xi_1$ and $\xi_2$ proved in Section~\ref{sec:proof_smooth_independence} are crucial in our study of the Lagrangian controllability.

\begin{prop}
    \label{prop:smooth_independence}
    Let $\omega \subset (1,\pi)$ be a nonempty open set and $0<\alpha_1<\alpha_2<1$. Then $\xi_1$ and $\xi_2$ defined through~\eqref{eq:model_linearized_adjoint} are bounded and smooth on $(0,T)\times (1,\pi)$. In particular, $f_1$ and $f_2$ are bounded and smooth on $(0,T)\times (0,\pi)$. Moreover, $\xi_1$ and $\xi_2$ are linearly independent in $L^2((0,T)\times \omega)$.
\end{prop}

In what follows, we fix $\eta>0$ small so that $\chi_{\eta}^{1/2}\xi_1$ and $\chi_{\eta}^{1/2}\xi_2$ are also linearly independent in $L^2((0,T)\times \omega)$. With these preparations, our main theorem reads:

\begin{thm}
  \label{thm:local_exact_Lagrangian_controllability}
  Let $T>0$ and $\omega \subset (1,\pi)$ be a nonempty open set. Then for any $0<\alpha_1<\alpha_2<1$, there exists $\delta=\delta(\alpha_1,\alpha_2,T,\omega)>0$ and a neighborhood $U$ of $(0,0)$ such that if $0<\beta_1<\beta_2<1$ satisfies $\max_{i=1,2}(|\alpha_i-\beta_i|)<\delta$, there exists a unique pair of real numbers $(\epsilon_1,\epsilon_2)$ in $U$ satisfying $\phi[\epsilon_1 f_1+\epsilon_2 f_2](T,\alpha_i)=\beta_i$ for $i=1,2$.
\end{thm}

Since the flow map preserves order, that is, $\phi(T,x_1)<\phi(T,x_2)$ if $0\leq x_1<x_2 \leq \pi$, the mapping of the end points $\phi(T,\alpha_i)=\beta_i$ ($i=1,2$) implies the transportation of the whole interval $\phi(T,I)=J$. Thus our main theorem shows that a remote action $f$ can move the fluid particles initially occupying the interval $I=[\alpha_1,\alpha_2]$ exactly onto $J=[\beta_1,\beta_2$] if both are close enough.

\begin{rem}
    \label{rem:several_intervals}
    We can generalize Theorem~\ref{thm:local_exact_Lagrangian_controllability} to more than two points: Let $0<\alpha_1<\alpha_2<\cdots<\alpha_d<1$ and $0<\beta_1<\beta_2<\cdots<\beta_d<1$ with $d\geq 2$. Define $f_i$ by~\eqref{eq:fi} for $i=1,2,\dotsc,d$. Then if $\max_{i=1,2,\dotsc,d}|\alpha_i-\beta_i|$ is sufficiently small, we can find $(\epsilon_1,\epsilon_2,\dotsc,\epsilon_d)$ such that the flow map associated to $f=\sum_{i=1}^{d}\epsilon_i f_i$ maps $\alpha_i$ to $\beta_i$ for every $i=1,2,\dotsc,d$ in a given  time $T>0$. The proof is almost identical except for an additional difficulty for which we use Lemma~\ref{conj:1} for general $d\geq 2$.
\end{rem}

\begin{rem}
    At present we have not been able to treat initial data other than $(1,0)$ or $(\bar{\rho},0)$ with $\bar{\rho}>0$. On the other hand, in the case of viscous Burgers' equation, Horsin was able to include any sufficiently regular initial data in~\cite[Theorem~5.1]{Horsin08}. If we were to adapt his strategy to treat general initial data $(\rho_0,u_0)$, we would have to consider the linearization of~\eqref{eq:model} around $(\tilde{\rho},\tilde{u})$---the solution to~\eqref{eq:model} with the initial data $(\rho_0,u_0)$ and the external force $f=0$. The resulting linear system is of variable coefficient, and our approach using Fourier analysis to prove the linear independence claim in Proposition~\ref{prop:smooth_independence} becomes difficult to apply.
\end{rem}

\subsection{Comparison with Eulerian controllability results}
It is interesting to compare our theorem with known \textit{Eulerian} controllability results. In~\cite{EGGP12}, Ervedoza, Glass, Guerrero, and Puel proved an exact Eulerian controllability result for~\eqref{eq:model} with initial data in a small neighborhood of $(\rho,u)=(\bar{\rho},\bar{u})$ with $\bar{\rho}>0$ and $\bar{u}\neq 0$. More precisely, they proved the following:

\begin{thm}[{\cite[Theorem~1.1]{EGGP12}}]
    \label{thm:Ervedoza}
    Let $L$ and $T$ be positive numbers. Assume that $\bar{\rho}>0$ and $\bar{u}\neq 0$ satisfy $T>L/|\bar{u}|$. Then there exists $\epsilon_0>0$ such that for any $(\rho_0,u_0)\in H^3(0,L)\times H^3(0,L)$ with 
    \begin{equation}
        \| u_0-\bar{u} \|_{H^3(0,L)}+\| \rho_0-\bar{\rho} \|_{H^3(0,L)}\leq \epsilon_0,
    \end{equation}
    there exists a solution $(\rho,u)$ to
    \begin{equation}
        \label{eq:model_EGGP}
        \begin{dcases}
            \rho_t+(\rho u)_x=0 & \text{in $(0,T)\times (0,L)$}, \\
            \rho(u_t+uu_x)+p(\rho)_x=u_{xx} & \text{in $(0,T)\times (0,L)$}, \\
            (\rho,u)(0,x)=(\bar{\rho},\bar{u}) & \text{for $x\in (0,L)$}
        \end{dcases}
    \end{equation}
    satisfying
    \begin{equation}
        (\rho,u)(T,\cdot)=(\bar{\rho},\bar{u}) \quad \text{in $(0,L)$}.
    \end{equation}
\end{thm}

\begin{rem}
    Note that the missing boundary conditions in~\eqref{eq:model_EGGP} play the role of controls, and in this sense, Theorem~\ref{thm:Ervedoza} is a boundary controllability result. We remark that our Theorem~\ref{thm:local_exact_Lagrangian_controllability} implies a boundary controllability result as well: Let $(\rho_f,u_f)$ be the solution to~\eqref{eq:model} with $f=\epsilon_1 f_1+\epsilon_2 f_2$ given in Theorem~\ref{thm:local_exact_Lagrangian_controllability}. Then $(\rho,u)=\restr{(\rho_f,u_f)}{(0,T)\times (0,1)}$ solves~\eqref{eq:model_EGGP} satisfying $u(t,0)=0$ for $t\in (0,T)$. Thus Theorem~\ref{thm:local_exact_Lagrangian_controllability} implies a boundary controllability result with controls $u(t,1)=u_f(t,1)$ for $t\in (0,T)$ and $\rho(t,1)=\rho_f(t,1)$ for $t\in (0,T)$ with $u_f(t,1)<0$ (note that boundary data for the density is needed only when there is an inflow at the boundary).
\end{rem}

In contrast with the theorem above, if $\bar{u}=0$, it is proved that the linearization of~\eqref{eq:model_EGGP} is not exactly Eulerian controllable~\cite{RR07,CRR12}.\footnote{We should remark that this does not exclude the controllability for the \textit{nonlinear} system as we might be able to control the nonlinear one using for example the return method of Coron~\cite{Coron93}. We also note that the result in~\cite{Amosova11} does not contradict these non-controllability results as she considers the equations in a Lagrangian coordinate.}~However, in these papers, it is also shown that the system is \textit{approximately} Eulerian controllable even when $\bar{u}=0$. Since Horsin showed in~\cite{Horsin08} that for viscous Burgers' equation, the local exact Lagrangian controllability can be deduced from the approximate Eulerian controllability---or rather its equivalent unique continuation property---it is natural to ask whether we can prove the local exact Lagrangian controllability for~\eqref{eq:model} as well. This is answered in the affirmative by Theorem~\ref{thm:local_exact_Lagrangian_controllability}.

\subsection{Comparison with controllability results for incompressible flows}
Controllability properties for incompressible flows are much well studied. Let us focus here on Lagrangian controllability results and we refer to~\cite{LSZ22} and the references therein for Eulerian ones.

First, the following global Lagrangian controllability result for 2D incompressible Euler equations are proved in~\cite{GH10}: Take two homotopic Jordan curves $\gamma_1$ and $\gamma_2$ bounding domains of the same area in a bounded domain $\Omega \subset \mathbb{R}^2$, and let $\Sigma$ be a nonempty open subset of $\partial \Omega$. Then for any $\epsilon>0$ and smooth initial data, there exists a solution to the 2D incompressible Euler equations with the no-flux boundary condition on $\partial \Omega \setminus \Sigma$ such that the corresponding flow map moves $\gamma_1$ to $\gamma_2$ within an error smaller than $\epsilon$ in a given time. This result was later extended to the 3D case in~\cite{GH12} (see also~\cite{HK17} for another approach) and to incompressible viscous flows in~\cite{LSZ22}.

All the works cited above use potential flows and the harmonic functions theory as basic building blocks of the proof. But since compressible potential flows are not directly related to harmonic functions and are hard to analyze, it seems difficult to employ these methods to compressible flows. Perhaps for this reason, at present, all the controllability results for compressible Navier--Stokes equations, including ours, are local ones. It remains an interesting future research direction to investigate global controllability theorems for compressible Euler or Navier--Stokes flows.

\section{Proof}
\label{sec:proof}
Let us briefly explain the strategy following~\cite{Horsin08}: Using $f_1$ and $f_2$ defined by~\eqref{eq:fi}, we introduce a map $\Theta \colon U' \to \mathbb{R}^2$ in a neighborhood $U'$ of $(0,0)$; this map is designed so that if we can apply the inverse function theorem to $\Theta$ around $(0,0)$, Theorem~\ref{thm:local_exact_Lagrangian_controllability} follows immediately. Therefore, the proof is reduced to proving the non-vanishing of the Jacobian of $\Theta$ at $(0,0)$. To prove this, the key point is the linear independence of $\xi_1$ and $\xi_2$ stated in Proposition~\ref{prop:smooth_independence}, which is essentially a unique continuation property for the adjoint linear system~\eqref{eq:model_linearized_adjoint}; see Remark~\ref{rem:unique_continuation_principle}. For the viscous Burgers equation considered in~\cite{Horsin08}, the adjoint linear system is the heat equation and the unique continuation property is easily obtained by Holmgren's theorem. For the compressible Navier--Stokes equations considered in this paper, we need a more subtle analysis. In the boundary controllability setting, this is showed in~\cite{RR07,CRR12} (in its equivalent approximate Eulerian controllability form) using Fourier analytic techniques. We adapt this technique to our distributed control setting. Let us pursue this strategy below.

\subsection{Jacobian as a Gram determinant}
\label{sec:proof_Jac}
First, let $U'$ be a sufficiently small neighborhood of $(0,0)\in \mathbb{R}^2$ and set
\begin{equation}
    \Theta \colon U'\ni (\epsilon_1,\epsilon_2)\mapsto (\phi[\epsilon_1 f_1+\epsilon_2 f_2](T,\alpha_1),\phi[\epsilon_1 f_1+\epsilon_2 f_2](T,\alpha_2)) \in \mathbb{R}^2
\end{equation}
where $f_1$ and $f_2$ are defined by~\eqref{eq:fi}. We remind the reader that $\phi[\epsilon_1 f_1+\epsilon_2 f_2]$ is the flow map associated to the external force $f=\epsilon_1 f_1+\epsilon_2 f_2$ and that the flow map is well-defined by Proposition~\ref{prop:wellposedness}. We also note that $\Theta$ is $C^1$ in a neighborhood of $(0,0)$ which follows from smooth dependence on $\epsilon_1$ and $\epsilon_2$ of solutions to~\eqref{eq:model} with $f=\epsilon_1 f_1+\epsilon_2 f_2$; here we need sufficient regularity of $f_i$, which is why we introduced a cutoff function in~\eqref{eq:fi}. So if we can show that the Jacobian of $\Theta$ at $(0,0)$ is non-zero, that is,
\begin{equation}
  \label{eq:J}
  J\coloneqq \det
  \begin{pmatrix}
    \restr{d\phi[\epsilon_1 f_1](T,\alpha_1)/d\epsilon_1}{\epsilon_1=0} & \restr{d\phi[\epsilon_1 f_1](T,\alpha_2)/d\epsilon_1}{\epsilon_1=0} \\
    \restr{d\phi[\epsilon_2 f_2](T,\alpha_1)/d\epsilon_2}{\epsilon_2=0} & \restr{d\phi[\epsilon_2 f_2](T,\alpha_2)/d\epsilon_2}{\epsilon_2=0}
  \end{pmatrix}
  \neq 0,
\end{equation}
it follows that there exists an open neighborhood $U$ of $(0,0)$ and $V$ of $\Theta(0,0)=(\alpha_1,\alpha_2)$ such that $\Theta$ induces a one-to-one map from $U$ to $V$. Note that this implies Theorem~\ref{thm:local_exact_Lagrangian_controllability}: If the interval $J=[\beta_1,\beta_2]$ is sufficient close to $I$, then $(\beta_1,\beta_2)\in V$, so we can find a unique $(\epsilon_1,\epsilon_2)\in U$ such that $\phi[\epsilon_1 f_1+\epsilon_2 f_2](T,\alpha_i)=\beta_i$ for $i=1,2$.

To show the non-vanishing of the Jacobian $J$, it turns out that it suffices to analyze the linearization of~\eqref{eq:model} around $(\rho,u)=(1,0)$. This is because Lemma~\ref{lem:derivative_in_terms_of_linearized_equations} below shows that $J$ can be computed in terms of the solution $(\eta,v)$ to
\begin{equation}
  \label{eq:model_linearized}
  \begin{dcases}
    \eta_t+v_x=g & \text{in $(0,T)\times (0,\pi)$}, \\
    v_t+c^2\eta_x=v_{xx}+f & \text{in $(0,T)\times (0,\pi)$}, \\
    v(t,0)=0, \quad v(t,\pi)=0 & \text{for $t\in (0,T)$}, \\
    (\eta,v)(0,x)=(0,0) & \text{for $x\in (0,\pi)$}
  \end{dcases}
\end{equation}
with $f=f_i$ and $g=0$ (here $c^2=p'(1)$). Before stating the lemma, let us note that by simple energy estimates, we can show that there exists $C_T>0$ such that
\begin{align}
    \label{eq:estimate_linear}
    \begin{aligned}
        & \sup_{0\leq t\leq T}\left( \| \eta(t,\cdot) \|_{L^2(0,\pi)}+\| v(t,\cdot) \|_{L^2(0,\pi)} \right)+\left( \int_{0}^{T}\| v_x(t,\cdot) \|_{L^2(0,\pi)}^{2}\, dt \right)^{1/2} \\
        & \quad \leq C_T \left( \| f \|_{L^2((0,T)\times (0,\pi))}+\| g \|_{L^2((0,T)\times (0,\pi))} \right).
    \end{aligned}
\end{align}

\begin{lem}
  \label{lem:derivative_in_terms_of_linearized_equations}
  For $i=1,2$, let $(\eta_i,v_i)$ be the solution to~\eqref{eq:model_linearized} with $f=f_i$ and $g=0$. Then we have
  \begin{equation}
    \restr{d\phi[\epsilon_i f_i](T,x)/d\epsilon_i}{\epsilon_i=0}=\int_{0}^{T}v_i(t,x)\, dt
  \end{equation}
  for all $x\in (0,\pi)$.
\end{lem}

\begin{proof}
  Fix $x\in (0,\pi)$. Let $(\rho_i,u_i)$ be the solution to~\eqref{eq:model} with $f=\epsilon_i f_i$. Note that
  \begin{equation}
    \label{eq:flow_map_formula}
    \phi[\epsilon_i f_i](t,x)=x+\int_{0}^{t}u_i(s,\phi[\epsilon_i f_i](s,x))\, ds \quad (0\leq t\leq T).
  \end{equation}
  Using this equality and applying the mean value theorem, with the help of Sobolev inequalities, we obtain
  \begin{align}
    & \int_{0}^{T}|u_i(t,\phi[\epsilon_i f_i](t,x))-u_i(t,x)|\, dt \\
    & \leq \int_{0}^{T}|\phi[\epsilon_i f_i](t,x)-x|\cdot \| \partial_x u_i(t,\cdot) \|_{L^{\infty}(0,\pi)}\, dt \\
    & \leq C\int_{0}^{T}\int_{0}^{t}|u_i(s,\phi[\epsilon_i f_i](s,x))|\, ds\cdot \| \partial_x u_i(t,\cdot) \|_{H^1(0,\pi)}\, dt \\
    & \leq C\int_{0}^{T}\int_{0}^{t}\| \partial_x u_i(s,\cdot) \|_{L^2(0,\pi)}\, ds\cdot \| \partial_x u_i(t,\cdot) \|_{H^1(0,\pi)}\, dt.
  \end{align}
  Then by Hölder's inequality as well as Proposition~\ref{prop:wellposedness}, we get
  \begin{align}
    & \int_{0}^{T}\int_{0}^{t}\| \partial_x u_i(s,\cdot) \|_{L^2(0,\pi)}\, ds\cdot \| \partial_x u_i(t,\cdot) \|_{H^1(0,\pi)}\, dt \\
    & \leq \int_{0}^{T}t^{1/2}\left( \int_{0}^{t}\| \partial_x u_i(s,\cdot) \|_{L^2(0,\pi)}^{2}\, ds \right)^{1/2}\cdot \| \partial_x u_i(t,\cdot) \|_{H^1(0,\pi)}\, dt \\
    & \leq T^{1/2}\| \partial_x u_i \|_{L^2(0,T;\, L^2(0,\pi))}\int_{0}^{T}\| \partial_x u_i(t,\cdot) \|_{H^1(0,\pi)}\, dt \\
    & \leq T\| \partial_x u_i \|_{L^2(0,T;\, H^1(0,\pi))}^{2}\leq CT\epsilon_{i}^{2}.
  \end{align}
  Next, since $(\eta,v)=(\rho_i-\epsilon_i \eta_i,u_i-\epsilon_i v_i)$ satisfies~\eqref{eq:model_linearized} with
  \begin{align}
      f & =-u_i\partial_x u_i-\frac{p'(\rho_i)-c^2(\rho_i-1)}{\rho_i}\partial_x \rho_i-\frac{\rho_i-1}{\rho_i}\partial_{xx}u_i-\frac{\rho_i-1}{\rho_i}\epsilon_i f_i, \\
      g & =-(\rho_i-1)\partial_x u_i-(\partial_x \rho_i)u_i,
  \end{align}
  Proposition~\ref{prop:wellposedness} and inequality~\eqref{eq:estimate_linear} imply
  \begin{align}
    \int_{0}^{T}|u_i(t,x)-\epsilon_i v_i(t,x)|\, dt
    & \leq CT^{1/2}\| \partial_x v \|_{L^2(0,T;\, L^2(0,\pi))} \\
    & \leq C_T \left( \| f \|_{L^2((0,T)\times (0,\pi))}+\| g \|_{L^2((0,T)\times (0,\pi))} \right) \leq C_T \epsilon_{i}^{2}.
  \end{align}
  Now using these inequalities to~\eqref{eq:flow_map_formula} with $t=T$, we conclude that
  \begin{equation}
    \phi[\epsilon_i f_i](T,x)=x+\int_{0}^{T}u_i(t,\phi[\epsilon_i f_i](t,x))\, dt=x+\epsilon_i \int_{0}^{T}v_i(t,x)\, dt+O(\epsilon_{i}^{2})
  \end{equation}
  as $\epsilon_i \to 0$. This proves the lemma.
\end{proof}

Now by a duality argument, we can write the Jacobian $J$ as the Gram determinant of $(\chi_{\eta}^{1/2}\xi_1,\chi_{\eta}^{1/2}\xi_2)\in L^2((0,T)\times \omega)$:

\begin{lem}
    \label{lem:J_is_Gram_det}
    The identity
    \begin{align}
        J
        & =\det \left( \int_{0}^{T}\int_{\omega}\chi_{\eta}\xi_i \xi_j\, dx\, dt \right)_{i,j=1,2}
    \end{align}
    holds for $J$ defined by~\eqref{eq:J}.
\end{lem}

\begin{proof}
    Define $(\eta_i,v_i)$ as in Lemma~\ref{lem:derivative_in_terms_of_linearized_equations}. Let
    \begin{equation}
        L=
        \begin{pmatrix}
            \partial_t & \partial_x \\
            c^2 \partial_x & \partial_t-\partial_{x}^{2}
        \end{pmatrix}
        \quad \text{and} \quad L^*=
        \begin{pmatrix}
            -\partial_t & -\partial_x \\
            -c^2 \partial_x & -\partial_t-\partial_{x}^{2}
        \end{pmatrix}.
    \end{equation}
    Note that
    \begin{equation}
        L
        \begin{pmatrix}
            \eta_i \\
            v_i
        \end{pmatrix}
        =
        \begin{pmatrix}
            0 \\
            f_i
        \end{pmatrix}
        \quad \text{and} \quad L^*
        \begin{pmatrix}
            \zeta_j \\
            \xi_j
        \end{pmatrix}
        =
        \begin{pmatrix}
            0 \\
            \delta_{\alpha_j}
        \end{pmatrix}.
    \end{equation}
    Then integration by parts yields
    \begin{align}
        \int_{0}^{T}v_i(t,\alpha_j)\, dt
        & =\int_{0}^{T}\int_{0}^{\pi}
        \begin{pmatrix}
            \eta_i \\
            v_i
        \end{pmatrix}
        \cdot
        \begin{pmatrix}
            0 \\
            \delta_{\alpha_j}
        \end{pmatrix}
        \, dx\, dt \\
        & =\int_{0}^{T}\int_{0}^{\pi}
        \begin{pmatrix}
            \eta_i \\
            v_i
        \end{pmatrix}
        \cdot L^*
        \begin{pmatrix}
            \zeta_j \\
            \xi_j
        \end{pmatrix}
        \, dx\, dt \\
        & =\int_{0}^{T}\int_{0}^{\pi}L
        \begin{pmatrix}
            \eta_i \\
            v_i
        \end{pmatrix}
        \cdot
        \begin{pmatrix}
            \zeta_j \\
            \xi_j
        \end{pmatrix}
        \, dx\, dt \\
        & =\int_{0}^{T}\int_{0}^{\pi}
        \begin{pmatrix}
            0 \\
            f_i
        \end{pmatrix}
        \cdot
        \begin{pmatrix}
            \zeta_j \\
            \xi_j
        \end{pmatrix}
        \, dx\, dt \\
        & =\int_{0}^{T}\int_{0}^{\pi}f_i\xi_j\, dx\, dt=\int_{0}^{T}\int_{\omega}\chi_{\eta}\xi_i\xi_j\, dx\, dt.
    \end{align}
    This and Lemma~\ref{lem:derivative_in_terms_of_linearized_equations} prove the desired result.
\end{proof}

\begin{rem}
    \label{rem:ODE_control}
    Our Lagrangian controllability problem can be recast into that for a system of ODEs: Let $u_{\epsilon_1 f_1+\epsilon_2 f_2}$ be the solution to~\eqref{eq:model} with $f=\epsilon_1 f_1+\epsilon_2 f_2$ and consider the ODE system
    \begin{equation}
        \begin{dcases}
            dx_1(t)/dt=u_{\epsilon_1 f_1+\epsilon_2 f_2}(t,x_1(t)) & (0\leq t\leq T), \\
            dx_2(t)/dt=u_{\epsilon_1 f_1+\epsilon_2 f_2}(t,x_2(t)) & (0\leq t\leq T), \\
            x_1(0)=\alpha_1, \quad x_2(0)=\alpha_2.
        \end{dcases}
    \end{equation}
    Then we try to find appropriate $\epsilon_1$ and $\epsilon_2$ so that $x_1(T)=\beta_1$ and $x_2(T)=\beta_2$ hold. If we linearize this system around $(\epsilon_1,\epsilon_2,x_1,x_2)=(0,0,\alpha_1,\alpha_2)$, we get
    \begin{equation}
        \label{eq:ODE_linearized}
        \begin{dcases}
            \frac{d}{dt}
            \begin{bmatrix}
                x_1 \\
                x_2
            \end{bmatrix}
            (t)
            =
            \begin{bmatrix}
                v_1(t,\alpha_1) & v_2(t,\alpha_1) \\
                v_1(t,\alpha_2) & v_2(t,\alpha_2)
            \end{bmatrix}
            \begin{bmatrix}
                \epsilon_1 \\
                \epsilon_2
            \end{bmatrix}
            & (0\leq t\leq T), \\
            x_1(0)=\alpha_1, \quad x_2(0)=\alpha_2
        \end{dcases}
    \end{equation}
    with $v_i$ defined in Lemma~\ref{lem:derivative_in_terms_of_linearized_equations}. The controllability Gramian (see~\cite[Definition~1.10]{Coron07}) of~\eqref{eq:ODE_linearized} is
    \begin{equation}
        \left( \int_{0}^{T}v_i(t,\alpha_j) \right)_{i,j=1,2}=\left( \int_{0}^{T}\int_{\omega}\chi_{\eta}\xi_i \xi_j \, dx\, dt \right)_{i,j=1,2}
    \end{equation}
    by the proof of Lemma~\ref{lem:J_is_Gram_det}. Therefore, the controllability of~\eqref{eq:ODE_linearized} is equivalent to the non-vanishing of our Jacobian $J$. Note that by the linear test~\cite[Theorem~3.6]{Coron07}, the controllability of the linearized system implies the nonlinear local controllability result as well.
\end{rem}

\subsection{Smoothness and linear independence of $\xi_1$ and $\xi_2$}
\label{sec:proof_smooth_independence}
By Lemma~\ref{lem:J_is_Gram_det} and the Cauchy--Schwarz inequality, Theorem~\ref{thm:local_exact_Lagrangian_controllability} is proved if we can show that $\xi_1$ and $\xi_2$ are linearly independent in $L^2((0,T)\times \omega)$ and $\eta>0$ is taken small enough. So it remains to prove Proposition~\ref{prop:smooth_independence}.

For this purpose, we first need an explicit formula for $\xi_i$:

\begin{lem}
  \label{lem:formula_of_xi_i}
  For $i=1,2$, let $(\zeta_i,\xi_i)$ be the solution to~\eqref{eq:model_linearized_adjoint}. Set $c=\sqrt{p'(1)}>0$ and
  \begin{equation}
    \label{eq:lambda_mu}
    \lambda_n=-\frac{n^2}{2}-\frac{n^2}{2}\sqrt{1-4c^2/n^2}, \quad \mu_n=-\frac{n^2}{2}+\frac{n^2}{2}\sqrt{1-4c^2/n^2} \quad (n\geq 1).
  \end{equation}
  Then if $n_0 \coloneqq 2c$ is not an integer, we have
  \begin{equation}
    \xi_i(t,x)=\frac{2}{\pi}\sum_{n\geq 1}\frac{\sin n\alpha_i}{\mu_n-\lambda_n}\left( e^{\mu_n(T-t)}-e^{\lambda_n(T-t)} \right) \sin nx.
  \end{equation}
  If $n_0$ is an integer, we have
  \begin{equation}
    \xi_i(t,x)=\frac{2}{\pi}q(t)\sin n_0 \alpha_i \sin n_0 x+\frac{2}{\pi}\sum_{n\geq 1,\, n\neq n_0}\frac{\sin n\alpha_i}{\mu_n-\lambda_n}\left( e^{\mu_n(T-t)}-e^{\lambda_n(T-t)} \right) \sin nx
  \end{equation}
  where
  \begin{equation}
    q(t)=\frac{1}{c^2}\left( 1-e^{-2c^2(T-t)} \right)-(T-t)e^{-2c^2(T-t)}.
  \end{equation}
\end{lem}

\begin{proof}
To calculate the solution to the non-homogeneous system~\eqref{eq:model_linearized_adjoint} by Duhamel's principle, we first need to compute the associated semigroup. Let us write $d_x=d/dx$ and define the operator $A\colon [L^2(0,\pi)]^2 \supset D(A)\to [L^2(0,\pi)]^2$ by
\begin{equation}
  D(A)=\left\{ (\zeta,\xi)^T \in [L^2(0,\pi)]^2 \relmiddle| \xi \in H_{0}^{1}(0,\pi),\, c^2\zeta+d_x \xi \in H^1(0,\pi) \right\}
\end{equation}
and
\begin{equation}
  A=
  \begin{pmatrix}
    0 & d_x \\
    c^2 d_x & d_{x}^{2}
  \end{pmatrix}.
\end{equation}
It can be shown that $-A$ is a maximal monotone operator in $[L^2(0,\pi)]^2$; see~\cite[Lemma~2.1]{CRR12}. Let us compute the semigroup $(S(t))_{t\geq 0}$ generated by $-A$ following~\cite{CRR12}. For this purpose, we introduce
\begin{gather}
  \Phi_0(x)=\frac{1}{\sqrt{c^2\pi}}(1,0)^T, \\
  \Phi_{2n}(x)=\sqrt{\frac{2}{c^2\pi}}(\cos nx,0)^T, \quad \Phi_{2n-1}(x)=\sqrt{\frac{2}{\pi}}(0,\sin nx)^T \quad (n\geq 1).
\end{gather}
These form an orthonormal basis of $[L^2(0,\pi)]^2$ with respect to the inner product
\begin{equation}
  \left\langle
    \begin{pmatrix}
      \zeta_1 \\
      \xi_1
    \end{pmatrix},
    \begin{pmatrix}
      \zeta_2 \\
      \xi_2
    \end{pmatrix}
  \right\rangle
  =c^2\int_{0}^{\pi}\zeta_1 \zeta_2 \, dx+\int_{0}^{\pi}\xi_1 \xi_2 \, dx.
\end{equation}
Note that the spaces
\begin{equation}
  V_0=\spn \{ \Phi_0 \}, \quad V_n=\spn \{ \Phi_{2n},\Phi_{2n-1} \} \quad (n\geq 1)
\end{equation}
are invariant under $A$, and the restriction of $A$ to $V_n$ ($n\geq 1$) has the matrix representation
\begin{equation}
  A_n=
  \begin{pmatrix}
    0 & cn \\
    -cn & -n^2
  \end{pmatrix}
\end{equation}
with respect to the basis $(\Phi_{2n},\Phi_{2n-1})$. If $n\neq n_0$, it can be diagonalized as
\begin{equation}
  A_n=P_n
  \begin{pmatrix}
    \lambda_n & 0 \\
    0 & \mu_n
  \end{pmatrix}
  P_{n}^{-1}
\end{equation}

and
\begin{equation}
  P_n=
  \begin{pmatrix}
    c & c \\
    \lambda_n/n & \mu_n/n
  \end{pmatrix}.
\end{equation}
Therefore, $S(t)$ has the matrix representation
\begin{equation}
    S_n(t)=e^{-A_n t}=P_n
    \begin{pmatrix}
        e^{-\lambda_n t} & 0 \\
        0 & e^{-\mu_n t}
    \end{pmatrix}
    P_{n}^{-1}
\end{equation}
on the space $V_n$. If $n=n_0$ on the other hand, $A_n$ can be triangulated as
\begin{equation}
  A_n=Q_n
  \begin{pmatrix}
    \lambda_n & 1 \\
    0 & \lambda_n
  \end{pmatrix}
  Q_{n}^{-1}
\end{equation}
with
\begin{equation}
  Q_n=
  \begin{pmatrix}
    c & -c/\lambda_n \\
    \lambda_n/n & 0
  \end{pmatrix},
\end{equation}
and $S(t)$ has the matrix representation
\begin{equation}
    S_n(t)=Q_n
    \begin{pmatrix}
        e^{-\lambda_n t} & te^{-\lambda_n t} \\
        0 & e^{-\lambda_n t}
    \end{pmatrix}
    Q_{n}^{-1}.
\end{equation}

Now, by Duhamel's principle, the solution $(\zeta_i,\xi_i)$ to~\eqref{eq:model_linearized_adjoint} can be written as
\begin{equation}
    \begin{pmatrix}
        \zeta_i \\
        \xi_i
    \end{pmatrix}
    (t)=\int_{t}^{T}S(t-\tau)
    \begin{pmatrix}
        0 \\
        \delta_{\alpha_i}
    \end{pmatrix}
    \, d\tau.
\end{equation}
So we get
\begin{align}
    \begin{pmatrix}
        \zeta_i \\
        \xi_i
    \end{pmatrix}
    (t,x)
    & =\int_{t}^{T}\left\langle
    \begin{pmatrix}
        0 \\
        \delta_{\alpha_i}
    \end{pmatrix}
    , \Phi_0 \right\rangle \Phi_0 \, d\tau \\
    & \quad +\sum_{n\geq 1}\int_{t}^{T}\left( \Phi_{2n},\Phi_{2n-1} \right)S_n(t-\tau)\left\{ \left\langle
    \begin{pmatrix}
        0 \\
        \delta_{\alpha_i}
    \end{pmatrix}
    , \Phi_{2n} \right\rangle \bm{e}_1+\left\langle
    \begin{pmatrix}
        0 \\
        \delta_{\alpha_i}
    \end{pmatrix}
    , \Phi_{2n-1} \right\rangle \bm{e}_2 \right\} \, d\tau \\
    & =\sum_{n\geq 1}\int_{t}^{T}\left( \Phi_{2n},\Phi_{2n-1} \right)S_n(t-\tau)\left\langle
    \begin{pmatrix}
        0 \\
        \delta_{\alpha_i}
    \end{pmatrix}
    , \Phi_{2n-1} \right\rangle \bm{e}_2 \, d\tau.
\end{align}
Here $\bm{e}_1=(1,0)^T$ and $\bm{e}_2=(0,1)^T$ are the standard basis of $\mathbb{R}^2$. Computing these terms explicitly, we get the desired formulas.
\end{proof}

The smoothness of $\xi_i$ on $(0,T)\times (1,\pi)$ can be proved using Lemma~\ref{lem:formula_of_xi_i}:

\begin{lem}
  \label{lem:smoothness}
  For $i=1,2$, let $(\zeta_i,\xi_i)$ be the solution to~\eqref{eq:model_linearized_adjoint}. Set $X_i=(0,\pi)\setminus \{ \alpha_i \}$. Then $\xi_i$ is bounded and smooth on $(0,T)\times X_i$.
\end{lem}

\begin{proof}
    The boundedness is clear from Lemma~\ref{lem:formula_of_xi_i} and the fact that $\mu_n-\lambda_n=O(n^2)$ as $n\to \infty$. Now note that $\lambda_n$ and $\mu_n$ are both real for all $n\geq n_0=2c$. By Lemma~\ref{lem:formula_of_xi_i} and $\lambda_n \leq -n^2/2$ for $n\geq n_0$, it suffices to show that
    \begin{equation}
        \sum_{n=n_0}^{\infty}\frac{\sin n\alpha_i}{\mu_n-\lambda_n}e^{\mu_n(T-t)}\sin nx=\sum_{n=n_0}^{\infty}\frac{1}{\mu_n-\lambda_n}e^{\mu_n(T-t)}\frac{1}{2}\left\{ \cos n(x-\alpha_i)-\cos n(x+\alpha_i) \right\}
    \end{equation}
    is smooth on $(0,T)\times X_i$. For this purpose, we introduce 
    \begin{equation}
        u(t,x)=\sum_{n=n_0}^{\infty}\frac{1}{\mu_n-\lambda_n}e^{\mu_n(T-t)}\cos nx.
    \end{equation}
    It is easy to see that $u$ is infinitely differentiable in $t$ on $(0,T)\times \mathbb{R}$ and that
    \begin{equation}
        \partial_{t}^{l}u(t,x)=\sum_{n=n_0}^{\infty}\frac{(-\mu_n)^{l}}{\mu_n-\lambda_n}e^{\mu_n(T-t)}\cos nx
    \end{equation}
    for any $l\geq 0$. To prove the lemma, it suffices to show that $\partial_{t}^{l}u$ is infinitely differentiable in $x$ on $(0,T)\times \{ (-2\pi,0)\cup (0,2\pi) \}$. Fix $t_0 \in (0,T)$. By~\eqref{eq:lambda_mu}, for any integer $K\geq 1$, there exist real numbers $a_k=a_k(t_0,l)$ ($k=1,2,\dotsc,K$) such that
    \begin{equation}
        \frac{(-\mu_n)^l}{\mu_n-\lambda_n}e^{\mu_n(T-t_0)}=\sum_{k=1}^{K}\frac{a_k}{n^{2k}}+O(n^{-2K-2}) \quad (n\to \infty).
    \end{equation}
    Note that
    \begin{equation}
        \sum_{n=n_0}^{\infty}O(n^{-2K-2})\cos nx
    \end{equation}
    is $2K$-times continuously differentiable on $\mathbb{R}$. On the other hand, it is well-known that\footnote{Extend the polynomial on the right-hand side from $[0,\pi]$ to $[-\pi,\pi]$ as an even function; the left-hand side is then its cosine expansion.}
     \begin{equation}
        \sum_{n=1}^{\infty}\frac{\cos nx}{n^2}=\frac{x^2}{4}-\frac{\pi x}{2}+\frac{\pi^2}{6} \quad (0\leq x\leq 2\pi).
    \end{equation}
    Therefore, the sum
    \begin{equation}
        \sum_{n=n_0}^{\infty}\frac{a_1}{n^2}\cos nx
    \end{equation}
    is infinitely differentiable on $(-2\pi,0)\cup (0,2\pi)$. By repeating termwise integration, we see that
    \begin{equation}
        \sum_{n=n_0}^{\infty}\frac{a_k}{n^{2k}}\cos nx
    \end{equation}
    is also infinitely differentiable on $(-2\pi,0)\cup (0,2\pi)$ for any $1\leq k\leq K$. Since $K\geq 1$ is arbitrary, the proof is complete.
\end{proof}

We finally prove the independence of $\xi_1$ and $\xi_2$ in $L^2((0,T)\times \omega)$:

\begin{lem}
    \label{lem:independence}
    The functions $\xi_1$ and $\xi_2$ defined through~\eqref{eq:model_linearized_adjoint} are linearly independent in $L^2((0,T)\times \omega)$.
\end{lem}

\begin{proof}
    Suppose that there exist $c_1,c_2 \in \mathbb{R}$ such that $c_1 \xi_1+c_2 \xi_2=0$ in $L^2((0,T)\times \omega)$. By Lemma~\ref{lem:formula_of_xi_i}, if $2c$ is not an integer, we have
    \begin{equation}
        \xi_i(t,x)=\frac{2}{\pi}\sum_{n\geq 1}\frac{\sin n\alpha_i}{\mu_n-\lambda_n}\left( e^{\mu_n(T-t)}-e^{\lambda_n(T-t)} \right)\sin nx
    \end{equation}
    for $(t,x)\in (0,T)\times \omega$, where $\lambda_n$ and $\mu_n$ are defined by~\eqref{eq:lambda_mu}. By the density of irrational numbers, there exists 
    $\gamma \in \mathbb{R}\setminus\mathbb{Q}$
    such that $\gamma \pi \in \omega$. Since $f_i$ is continuous on $(0,T)\times \omega$ by Lemma~\ref{lem:smoothness}, we have
    \begin{equation}
        \sum_{n\geq 1}\frac{c_1 \sin n\alpha_1+c_2 \sin n\alpha_2}{\mu_n-\lambda_n}\left( e^{\mu_n(T-t)}-e^{\lambda_n(T-t)} \right)\sin n\gamma \pi=0.
    \end{equation}
    Note that $\sin n\gamma \pi \neq 0$ for $n\geq 1$. As $\mu_n-\lambda_n=O(n^2)$, the series is absolutely convergent, and we can apply Lemma~\ref{lem:Rosier_Rouchon_variation} to conclude that $c_1 \sin n\alpha_1+c_2 \sin n\alpha_2=0$ for $n\geq 1$. By Lemma~\ref{conj:1}, this implies $c_1=c_2=0$. The case when $2c$ is an integer can be treated in a similar manner.
\end{proof}

\begin{rem}
    \label{rem:unique_continuation_principle}
    By a similar reasoning to that in the proof of Proposition~\ref{prop:smooth_independence}, we can prove the following unique continuation property: Let $(\zeta,\xi)$ be the solution to
    \begin{equation}
        \begin{dcases}
            -\zeta_t-\xi_x=0 & \text{in $(0,T)\times (0,\pi)$}, \\
            -\xi_t-c^2 \zeta_x=\xi_{xx} & \text{in $(0,T)\times (0,\pi)$}, \\
            \xi(t,0)=0, \quad \xi(t,\pi)=0 & \text{for $t\in (0,T)$}, \\
            (\zeta,\xi)(T,x)=(\zeta^T,\xi^T)(x) & \text{for $x\in (0,\pi)$}
        \end{dcases}
    \end{equation}
    with $(\zeta^T,\xi^T)\in H^1(0,\pi)\times L^2(0,\pi)$ and $\int_{0}^{\pi}\zeta^T \, dx=0$. Then if $\xi \equiv 0$ on $(0,T)\times \omega$, we must have $(\zeta,\xi)\equiv (0,0)$ on $(0,T)\times (0,\pi)$. We note that the condition $\xi \equiv 0$ on $(0,T)\times \omega$ can be replaced by $\xi(t,\gamma \pi)=0$ for $0<t<T$, where $\gamma \in (0,1)$ is an arbitrary irrational number. Thus the independence property in Proposition~\ref{prop:smooth_independence}, which is the key point in the proof of Theorem~\ref{thm:local_exact_Lagrangian_controllability}, is essentially a unique continuation property for the linearized adjoint system.
\end{rem}

Proposition~\ref{prop:smooth_independence} is obtained by combining Lemmas~\ref{lem:smoothness} and~\ref{lem:independence}. Then taking $\eta>0$ small enough, Propositions~\ref{prop:smooth_independence} and~\ref{lem:J_is_Gram_det} imply that the Jacobian $J$ is non-zero, which in turn proves Theorem~\ref{thm:local_exact_Lagrangian_controllability} thanks to the inverse function theorem.

\appendix
\renewcommand*{\thesection}{\Alph{section}}

\section{Generalization of a lemma by Rosier and Rouchon}
\label{appendix:Rosier_Rouchon}
We state here a generalization of~\cite[Lemma~1]{RR07}. The purpose of the generalization is to take care of the case when $n_0=2c$ is an integer in the proof of Lemma~\ref{lem:independence}.

\begin{lem}
\label{lem:Rosier_Rouchon_variation}
  Let $(c_n)_{n\geq -1}$ and $(\lambda_n)_{n\geq 0}$ be two sequences of complex numbers with $\lambda_n \neq \lambda_m$ for $n\neq m$. Assume that $\lambda_0$ is real and $(\lambda_n)_{n \geq 1}$ are real for all but finitely many $n\geq 1$; that $\sum_{n\geq -1}|c_n|<\infty$; and that there exists $\Lambda \in \mathbb{R}$ such that $\Re \lambda_n \leq \Lambda$ for all $n\geq 0$. Then if
  \begin{equation}
    (c_{-1}t+c_0)e^{\lambda_0 t}+\sum_{n\geq 1}c_n e^{\lambda_n t}=0 \quad (0<t<T)
  \end{equation}
  for some positive number $T>0$, then $c_n=0$ for all $n\geq -1$.
\end{lem}

\begin{proof}
  Let
  \begin{equation}
    F(z)=(c_{-1}z+c_0)e^{\lambda_0 z}+\sum_{n\geq 1}c_n e^{\lambda_n z}.
  \end{equation}
  By the assumptions of the lemma, $F$ is analytic on the right half-plane. Since $F(t)=0$ for $0<t<T$, by the unique continuation property for analytic functions, it follows that $F(z)=0$ for every $z$ in the right half-plane. Let $I_+$ be the finite set of $n\geq 1$ such that the imaginary part of $\lambda_n$ is positive. Since $F(1-it)=0$ for $t\in \mathbb{R}$, we have
  \begin{align}
    \label{eq:F(1-it)=0}
    \sum_{n\in I_+}c_n e^{\lambda_n(1-it)}
    & =-\{ c_{-1}(1-it)+c_0 \} e^{\lambda_0(1-it)}-\sum_{n\geq 1;\, n\notin I_+}c_n e^{\lambda_n(1-it)}
  \end{align}
  for $t\in \mathbb{R}$. Now let $\lambda_n=\mu_n+i\nu_n$ ($\mu_n,\nu_n \in \mathbb{R}$) and set $\nu=\max_{n\in I_+}\nu_n$. Then by~\eqref{eq:F(1-it)=0} and $\sum_{n\geq 1}|c_n|<\infty$, it follows that
  \begin{align}
    p(t)
    & \coloneqq \sum_{n\in I_+;\, \nu_n=\nu}c_n e^{\lambda_n}e^{-i\mu_n t} \\
    & =-e^{-\nu t}\sum_{n\in I_+;\, \nu_n<\nu}c_n e^{\lambda_n}e^{-i\mu_n t}e^{\nu_n t} \\
    & \quad-e^{-\nu t}\left[ \{ c_{-1}(1-it)+c_0 \} e^{\lambda_0(1-it)}+\sum_{n\geq 1;\, n\notin I_+}c_n e^{\lambda_n(1-it)} \right]
  \end{align}
  tends to zero as $t\to \infty$.\footnote{Note that $\nu>0$ if $I_+$ is non-empty; if it is empty, then $p(t)=0$.}~This implies that $p(t)=0$ for all $t\in \mathbb{R}$: Since $p(t)$ is an almost periodic function~\cite[Theorem~III]{Bohr47}, we can find, for any $\epsilon>0$, an arbitrarily large $\tau=\tau(\epsilon)$ independent of $t$ such that
  \begin{equation}
    |p(t)-p(t+\tau)|\leq \epsilon \quad (t\in \mathbb{R}).
  \end{equation}
  Hence, by taking $\tau$ sufficiently large, $|p(t)|\leq |p(t+\tau)|+\epsilon \leq 2\epsilon$, which means that $p(t)=0$. Now by taking the Laplace transform of $p(t)$, we get
  \begin{equation}
      \sum_{n\in I_+;\, \nu_n=\nu}\frac{c_n e^{\lambda_n}}{s+i\mu_n}=0 \quad (\Re s>0).
  \end{equation}
  Then it is easy to see that $c_n=0$ for $n\in I_+$ with $\nu_n=\nu$.\footnote{For $n,m\in I_+$ with $\nu_n=\nu_m=\nu$, we have $\mu_n \neq \mu_m$ when $n\neq m$.}~Arguing similarly, it follows that $c_n=0$ for any $n\in I_+$. Considering also the behavior as $t\to -\infty$, we get $c_n=0$ for all $n$ such that the imaginary part of $\lambda_n$ is negative.
  
  From the argument above, we may assume that $(\lambda_n)_{n\geq 0}$ are real for all $n\geq 0$. Now write $c_{-1}'=e^{\lambda_0}c_{-1}$ and $c_{n}'=e^{\lambda_n}c_n$ for $n\geq 0$. By taking the average of the relation $tF(1+it)e^{-i\lambda_0 t}=0$ over the interval $[-T,T]$, we get
  \begin{align}
    ic_{-1}'T^2/3
    & =\frac{1}{2T}\int_{-T}^{T}t\left\{ c_{-1}'(1+it)+c_{0}' \right\} \, dt=-\sum_{n\geq 1}\frac{1}{2T}\int_{-T}^{T}c_{n}'te^{i(\lambda_n-\lambda_0)t}\, dt.
  \end{align}
  By the assumption $\sum_{n\geq -1}|c_n|<\infty$, for any $\delta>0$, there exists a sufficiently large integer $M$ such that
  \begin{align}
      |c_{-1}'|
      & \leq \sum_{1\leq n\leq M}\frac{3}{2T^3}\left| \int_{-T}^{T}c_{n}'te^{i(\lambda_n-\lambda_0)t}\, dt \right|+\delta.
  \end{align}
  Since
  \begin{equation}
      \frac{3}{2T^3}\int_{-T}^{T}te^{i\lambda t}\, dt=\frac{3\cos \lambda T}{i\lambda T^2}-\frac{3\sin \lambda T}{i\lambda^2 T^3}
  \end{equation}
  for $\lambda \in \mathbb{R}\setminus \{ 0 \}$, by taking $T>0$ sufficient large, we get $|c_{-1}'|\leq 2\delta$. Hence $c_{-1}=0$.
  
  The rest of the proof is the same as that of~\cite[Lemma~1]{RR07}; we reproduce it here for the sake of completeness. By taking the average of the relation $F(1+it)e^{-i\lambda_N t}=0$ over the interval $[-T,T]$, we get
  \begin{equation}
      c_{-N}'=-\sum_{n\geq 0;\, n\neq N}\frac{1}{2T}\int_{-T}^{T}c_{n}'e^{i(\lambda_n-\lambda_N)t}\, dt
  \end{equation}
  for any $N\geq 0$. For any $\delta>0$, there exists a sufficiently large integer $M$ such that
  \begin{align}
      |c_{N}'|
      & \leq \sum_{0\leq n\leq M;\, n\neq N}\frac{1}{2T}\left| \int_{-T}^{T}c_{n}'e^{i(\lambda_n-\lambda_N)t}\, dt \right|+\delta.
  \end{align}
  Since
  \begin{equation}
      \frac{1}{2T}\int_{-T}^{T}e^{i\lambda t}\, dt=\frac{\sin \lambda T}{\lambda T}
  \end{equation}
  for $\lambda \in \mathbb{R}\setminus \{ 0 \}$, by taking $T>0$ sufficient large, we get $|c_{N}'|\leq 2\delta$. Hence,  $c_{N}=0$. This completes the proof.
\end{proof}

\section{Linear independence of some trigonometric vectors}
\label{appendix:independence_sin}

In the proof of Lemma~\ref{lem:independence}, we used the fact that $c_1 \sin n\alpha_1+c_2 \sin n\alpha_2=0$ for $n\geq 1$ implies $c_1=c_2=0$. We give a proof of this in the following. In fact, to prove the extension mentioned in Remark~\ref{rem:several_intervals}, we prove a more general lemma:

\begin{lem}
    \label{conj:1}
    Let $0<\alpha_1<\alpha_2<\cdots<\alpha_d<\pi$ and let $c_1,c_2,\dotsc,c_d$ be real numbers. If
    \begin{equation}
        c_1 \sin(n\alpha_1)+c_2 \sin(n\alpha_2)+\cdots+c_d \sin(n\alpha_d)=0
    \end{equation}
    holds for $n=1,2,\dotsc,d$, then $c_1=c_2=\cdots=c_d=0$.
\end{lem}

To prove this, it suffices to show that
\begin{equation}
    D=D(\alpha_1,\alpha_2,\dotsc,\alpha_d)\coloneqq \det
    \begin{pmatrix}
        \sin \alpha_1  & \cdots & \sin \alpha_i  & \cdots & \sin \alpha_d \\
        \vdots         & \ddots &                &        & \vdots \\
        \sin i\alpha_1 &        & \sin i\alpha_i &        & \sin i\alpha_d \\
        \vdots         &        &                & \ddots & \vdots \\
        \sin d\alpha_1 & \cdots & \sin d\alpha_i & \cdots & \sin d\alpha_d
    \end{pmatrix}
    \neq 0.
\end{equation}
This is a consequence of the following lemma:

\begin{lem}
    \label{conj:2}
    Let $0<\alpha_1<\alpha_2<\cdots<\alpha_d<\pi$. Then the following identity holds:
    \begin{equation}
        D(\alpha_1,\alpha_2,\dotsc,\alpha_d)=2^{d(d-1)}\left( \prod_{i=1}^{d}\sin \alpha_i \right) \prod_{1\leq i<j\leq d}\sin \left( \frac{\alpha_i-\alpha_j}{2} \right) \sin \left( \frac{\alpha_i+\alpha_j}{2} \right).
    \end{equation}
\end{lem}

\begin{proof}
    Let $1\leq i\leq d$. By the multiple-angle formula
    \begin{equation}
        \label{eq:multiple_angle}
        \sin(i\theta)=\sin \theta \sum_{k=0}^{\lfloor (i-1)/2 \rfloor}(-1)^k \binom{i}{2k+1}(1-\cos^2 \theta)^k \cos^{i-2k-1}\theta
    \end{equation}
    and an elementary identity involving binomial coefficients, it follows that there exists a polynomial $S_i(x)$ that satisfies:
    \begin{enumerate}[label=(\roman*)]
        \item $\sin(i\theta)=S_i(\cos \theta)\sin \theta$,
        \item $\deg S_i(x)=i-1$,
        \item the leading coefficient of $S_i(x)$ is $2^{i-1}$.
    \end{enumerate}
    Now by (i)--(iii) and some elementary row operations,
    \begin{align}
        D
        & =\det
        \begin{pmatrix}
            \sin \alpha_1  & \cdots & \sin \alpha_i  & \cdots & \sin \alpha_d \\
            \vdots         & \ddots &                &        & \vdots \\
            S_i(\cos\alpha_1)\sin \alpha_1 &        & S_i(\cos\alpha_i)\sin \alpha_i &        & S_i(\cos\alpha_d)\sin \alpha_d \\
            \vdots         &        &                & \ddots & \vdots \\
            S_d(\cos\alpha_1)\sin \alpha_1 & \cdots & S_d(\cos\alpha_i)\sin \alpha_i & \cdots & S_d(\cos\alpha_d)\sin \alpha_d
        \end{pmatrix} \\ 
        & = \left( \prod_{i=1}^d \sin \alpha_i \right) \det
        \begin{pmatrix}
            1  & \cdots & 1 & \cdots & 1 \\
            \vdots         & \ddots &                &        & \vdots \\
            S_i(\cos\alpha_1) &        & S_i(\cos\alpha_i) &        & S_i(\cos\alpha_d) \\ 
            \vdots         &        &                & \ddots & \vdots \\
            S_d(\cos\alpha_1) & \cdots & S_d(\cos\alpha_i) & \cdots & S_d(\cos\alpha_d)
        \end{pmatrix} \\
         & =\left( \prod_{i=1}^d \sin \alpha_i \right) \det 
        \begin{pmatrix}
            1  & \cdots & 1 & \cdots & 1 \\
            \vdots         & \ddots &                &        & \vdots \\
            2^{i-1}\cos^{i-1}\alpha_1 &        & 2^{i-1}\cos^{i-1}\alpha_i &        & 2^{i-1}\cos^{i-1}\alpha_d \\
            \vdots         &        &                & \ddots & \vdots \\
            2^{d-1}\cos^{d-1}\alpha_1 & \cdots & 2^{d-1}\cos^{d-1}\alpha_i & \cdots & 2^{d-1}\cos^{d-1}\alpha_d
        \end{pmatrix}.
    \end{align}
    The determinant in the last line is a Vandermonde determinant. So we have
    \begin{align}
        D
        &=\left( \prod_{i=1}^d \sin \alpha_i \right) 2^{d(d-1)/2}\prod_{1\leq i<j\leq d}(\cos\alpha_j-\cos\alpha_i)\\
        &= 2^{d(d-1)}\left( \prod_{i=1}^d \sin \alpha_i \right) \prod_{1\leq i< j\leq n}\sin\left(\frac{\alpha_i-\alpha_j}{2}\right)\sin\left(\frac{\alpha_i+\alpha_j}{2}\right).
    \end{align}
    This proves the lemma.
\end{proof}

\subsubsection*{Acknowledgements}
K. Koike was supported by JSPS Grant-in-Aid for Early-Career Scientists (Grant Number 22K13938) and Grant-in-Aid for Pioneering Research (Organization for Fundamental Research, Tokyo Institute of Technology). F. Sueur was supported by the Bourgeons project, grant ANR-23-CE40-0014-01 of the French National Research Agency (ANR). G. Vergara-Hermosilla was supported by the ANID program BCH 2022 grant No. 4220003. We thank Kazuaki Miyatani for giving a nice proof of Lemma~\ref{conj:2}. We also acknowledge Lorenzo Cavallina for his suggestion to consider the generalization to more than two points (Remark~\ref{rem:several_intervals}). Finally, we are grateful to Sylvain Ervedoza for several helpful discussions.

\subsubsection*{Data availability statement}
Data sharing is not applicable to this article as no datasets were generated or analyzed during the current study.

\subsubsection*{Declarations}
The authors declare that they have no conflicts of interest.

\bibliographystyle{amsplain}
\bibliography{note}

\end{document}